\documentclass[12pt]{amsart}

\usepackage{etex}
\usepackage{amsmath, amssymb}
\usepackage{array}
\usepackage[frame,cmtip,arrow,matrix,line,graph,curve]{xy}
\usepackage{graphpap, color, paralist, pstricks}
\usepackage[mathscr]{eucal}
\usepackage[pdftex]{graphicx}
\usepackage[pdftex,colorlinks,backref=page,citecolor=blue]{hyperref}
\usepackage{tikz}

\newtheorem{theorem}{Theorem}[section]
\newtheorem{proposition}[theorem]{Proposition}
\newtheorem{corollary}[theorem]{Corollary}
\newtheorem{lemma}[theorem]{Lemma}

\newtheorem{conjecture}[theorem]{Conjecture}

\theoremstyle{definition}
\newtheorem{definition}[theorem]{Definition}
\newtheorem{example}[theorem]{Example}

\newtheorem{remark}[theorem]{Remark}

\newcommand{\PP}{\mathbb{P}}
\newcommand{\QQ}{\mathbb{Q}}

\newcommand{\RR}{\mathbb{R}}
\newcommand{\ZZ}{\mathbb{Z}}

\newcommand{\cO}{\mathcal{O} }

\newcommand{\cM}{\mathcal{M} }

\newcommand{\rH}{\mathrm{H} }

\newcommand{\rN}{\mathrm{N} }

\newcommand{\spec}{\mathrm{Spec}\;}

\newcommand{\Mznb}{\overline{\mathrm{M}}_{0,n}}

\def\SL{\mathrm{SL}}

\def\git{/\!/ }
\def\nef{\mathrm{Nef} }
\def\eff{\mathrm{Eff} }
\def\checkedn{19}
\def\bpfcheckedn{16}

\hypersetup{%
pdftitle={On the $S_{n}$-invariant F-conjecture},%
pdfauthor={Han-Bom Moon, David Swinarski},%
pdfkeywords={moduli space, rational curves, birational geometry, nef cone, base point free divisor},%
citecolor=blue,%
linkcolor=blue,%
}

\begin{document}

\title{On the $S_{n}$-invariant F-conjecture}
\date{\today}

\author{Han-Bom Moon}
\address{Department of Mathematics, Fordham University, Bronx, NY 10458}
\email{hmoon8@fordham.edu}

\author{David Swinarski}
\address{Department of Mathematics, Fordham University, New York, NY 10023}
\email{dswinarski@fordham.edu}

\begin{abstract}
By using classical invariant theory, we reduce the $S_{n}$-invariant F-conjecture to a feasibility problem in polyhedral geometry. We show by computer that for $n \le \checkedn$, every integral $S_{n}$-invariant F-nef divisor on the moduli space of genus zero stable pointed curves is semi-ample, over arbitrary characteristic. Furthermore, for $n \le \bpfcheckedn$, we show that for every integral $S_{n}$-invariant nef (resp. ample) divisor $D$ on the moduli space, $2D$ is base-point-free (resp. very ample). As applications, we obtain the nef cone of the moduli space of stable curves without marked points, and the semi-ample cone that of the moduli space of genus 0 stable maps to Grassmannian for small numerical values. 
\end{abstract}

\maketitle

%%%%%%%%%%%%%%%%%%%%%%%%%%%%%%%%%%%

\section{Introduction}

When one studies birational geometric aspects of a projective variety $X$, the first step is to understand two cones of divisors in $\mathrm{N}^{1}(X)$: the effective cone $\eff(X)$ and the nef cone $\nef(X)$. The first cone contains information on rational contractions of $X$, and the second cone contains data on regular contractions of $X$. 

In this paper, we study the moduli space $\Mznb$ of genus 0 stable
pointed curves. Its elementary construction in \cite{Kap93b} suggests
that its geometry is similar to that of toric varieties, and therefore
many people conjectured that $\eff(\Mznb)$ and $\nef(\Mznb)$ are
polyhedral. However, the birational geometric properties of $\Mznb$
seem to be very complicated. The cone of effective divisors was
conjectured to be generated by boundary divisors. However, now there
are many known examples of non-boundary extremal effective divisors
(\cite{Ver02, CT13, Opi16}). Doran, Giansiracusa, and Jensen showed
that the effectivity of a divisor class depends on the base ring, and
there are generators of the Cox ring which do not lie on extremal rays of $\eff(\Mznb)$ (\cite{DGJ14}). Furthermore, recently it was shown that $\Mznb$ is not a Mori dream space for $n \ge 10$ (\cite{CT15, GK16, HKL16}). On the other hand, the $S_{n}$-invariant part $\eff(\Mznb)^{S_{n}}$ is simplicial and generated by symmetrized boundary divisors $\{B_{i} = \sum_{|I| = i}B_{I}\}$ (\cite[Theorem 1.3]{KM13}). 

For $\nef(\Mznb)$, there has been less progress, but there is an explicit conjectural description. From the analogy with toric varieties again, a natural candidate of the generating set of the Mori cone of $\Mznb$ is the set of one-dimensional boundary strata, called \emph{F-curves}. 

\begin{definition}\label{def:Fnef}
An effective divisor $D = \sum b_{I}B_{I}$ on $\Mznb$ is \emph{F-nef} if for any F-curve $F$, $D \cdot F \ge 0$. 
\end{definition}

Although this definition uses intersection theory, we can explicitly state the set of linear inequalities with respect to the coefficients $\{b_{I}\}$ (\cite[(0.14)]{GKM02}). Thus we can formally define F-nefness over $\spec \ZZ$, as well.

Fulton conjectured that the Mori cone of $\Mznb$ is generated by F-curves. Dually:

\begin{conjecture}[F-conjecture]\label{conj:Fconj}
A divisor on $\Mznb$ is nef if and only if it is F-nef.
\end{conjecture}

This conjecture was shown for $n \le 7$ in \cite{KM13} by Keel and McKernan in characteristic 0. However, as $n$ grows, the Picard number of $\Mznb$ grows exponentially, so $\overline{\mathrm{M}}_{0,8}$ is already out of reach. 

On the other hand, we may ask the same question for $S_{n}$-invariant divisors:

\begin{conjecture}[$S_{n}$-invariant F-conjecture]\label{conj:SnFconj}
An $S_{n}$-invariant divisor on $\Mznb$ is nef if and only if it is F-nef.
\end{conjecture}

In characteristic 0, Gibney proved Conjecture \ref{conj:SnFconj} for $n \le 24$ (\cite{Gib09}). Recently, Fedorchuk showed that the $S_{n}$-invariant F-conjecture is true for $n \le 16$ in arbitrary characteristic (\cite{Fed14b}). 

\subsection{Results}

In this paper, we translate Conjecture \ref{conj:SnFconj} into a feasibility problem in polyhedral geometry, namely, the nonemptiness of certain polytopes. We use this approach to show the following result. 

\begin{theorem}[Theorem \ref{thm:semiamplecomputation}]\label{thm:mainthm}
\begin{enumerate}
\item For $n \le \checkedn$, over $\spec \ZZ$, every $S_{n}$-invariant F-nef divisor on $\Mznb$ is semi-ample. 
\item For $n \le \bpfcheckedn$, over $\spec \ZZ$, for every $S_{n}$-invariant F-nef divisor on $\Mznb$, $2D$ is base-point-free.
\end{enumerate} 
\end{theorem}

By using \cite{KT09} or \cite{Tev07}, we obtain the following consequence. 

\begin{theorem}[Theorem \ref{thm:veryample}]\label{thm:veryampleintro}
Suppose $n \le \bpfcheckedn$. Over any algebraically closed field, for every integral $S_{n}$-invariant ample divisor $A$ on $\Mznb$, $2A$ is very ample.
\end{theorem}

We immediately obtain the following two corollaries. 

\begin{corollary}
For $n \le \checkedn$, over any algebraically closed field, $\nef(\Mznb/S_{n})$ is equal to the F-nef cone and every nef divisor on $\Mznb/S_{n}$ is semi-ample.
\end{corollary}

\begin{corollary}
For $n \le \checkedn$, over any algebraically closed field, the Mori cone of $\Mznb/S_{n}$ is generated by F-curves. 
\end{corollary}

This result also yields the description of the nef cone of some other moduli spaces. By \cite[Theorem 0.3]{GKM02}, we obtain the nef cone of $\overline{\cM}_{g}$, whose description was a question raised by Mumford. 

\begin{corollary}
Over any algebraically closed field, for $g \le \checkedn$, the nef cone of $\overline{\mathcal{M}}_{g}$, the moduli space of genus $g$ stable curves, is equal to the F-nef cone. 
\end{corollary}

By \cite[Theorem 1.1]{CHS09} and \cite[Theorem 1.1]{CS06}, we obtain the nef cone of the moduli space of genus 0 stable maps to a Grassmannian, including the case of projective space. 

\begin{corollary}
Let $k$, $n$, and $d$ be positive integers such that $1 \le k \le n-1$ and $d \le \checkedn$. Let $\overline{\mathrm{M}}_{0,0}(\mathrm{Gr}(k, n), d)$ be the moduli space of genus 0 stable maps to the Grassmannian $\mathrm{Gr}(k, n)$. Over any algebraically closed field, $\nef(\overline{\mathrm{M}}_{0,0}(\mathrm{Gr}(k, n), d))$ coincides with the semi-ample cone, and it is polyhedral with explicit generators. 
\end{corollary}

\subsection{Idea of the proof}

The main ingredient of the proof of Theorem \ref{thm:mainthm} is classical invariant theory, in particular \emph{graphical algebras}. 

A simple but crucial observation in this paper is that an $S_{n}$-invariant F-nef divisor $D$ can be written as $\pi^{*}(cD_{2}) - \sum_{i \ge 3}a_{i}B_{i}$ where $\pi : \Mznb \to (\PP^{1})^{n}\git \SL_{2}$ is a regular contraction and $D_{2}$ is an ample divisor. Then the linear system $|D|$ can be identified with a sub linear system $|cD_{2}|_{\mathbf{a}}$ of $|cD_{2}|$ on $(\PP^{1})^{n}\git \SL_{2}$ (Proposition \ref{prop:generatorsofD}). Thus the study of $|D|$ can be reduced to the study of a non-complete linear system on $(\PP^{1})^{n}\git \SL_{2}$. Furthermore, there is a sub linear system $|cD_{2}|_{\mathbf{a}, G} \subset |cD_{2}|_{\mathbf{a}}$ which can be described combinatorially via \emph{graphical algebra}. The Cox ring of $(\PP^{1})^{n}\git \SL_{2}$, which is the ring of $\SL_{2}$-invariant divisors on $(\PP^{1})^{n}$, is classically known as the graphical algebra since the 19th century. Its generators can be described in terms of finite graphs, thus we may study it by using graph theory. By using the graphical algebra, we obtain a combinatorial description of $|cD_{2}|_{\mathbf{a}, G}$ and its base locus in terms of polytopes (Corollary \ref{cor:polyhedralcone}). 

\subsection{Structure of the paper}

This paper is organized as the following. In Section \ref{sec:graphicalalgebra}, we recall the definition of the graphical algebra. In Section \ref{sec:Gbpf}, we translate the F-conjecture into a feasibility problem. Section \ref{sec:computation} presents the proof of the main theorem, computational results, and some examples.

\subsection*{Acknowledgements}

We would like to thank Maksym Fedorchuk, who kindly pointed out a crucial error on the earlier draft of this paper, and gave many invaluable suggestions. We also thank Young-Hoon Kiem and Ian Morrison for many suggestions and comments.

%%%%%%%%%%%%%%%%%%%%%%%%%%%%%%%%%%%%%

\section{Graphical algebra}\label{sec:graphicalalgebra}

In this section, we recall the definition and basic properties of graphical algebras, which are key algebraic tools in this approach to the $S_{n}$-invariant F-conjecture. We work over $\ZZ$, but all of the results stated here are valid over any base ring. 

The \emph{graphical algebra} is a $\ZZ$-algebra which was introduced to describe a result in classical invariant theory. Consider $(\PP^{1})^{n}$, the space of $n$ points on a projective line. There is a natural diagonal $\SL_{2}$-action on this space, which is induced by a homomorphism $\SL_{2} \to \mathrm{PGL}_{2} \cong \mathrm{Aut}(\PP^{1})$. The graphical algebra is the ring of $\SL_{2}$-invariant multi-homogeneous polynomials. 

Let $[X_{i}: Y_{i}]$ be the homogeneous coordinates of $i$-th factor of $(\PP^{1})^{n}$. It is straightforward to check that $Z_{ij} := (X_{i}Y_{j}-X_{j}Y_{i})$ is an $\SL_{2}$-invariant polynomial, and their products are all invariant polynomials. We can index such polynomials by using finite digraphs. Let $\overrightarrow{\Gamma}$ be a finite directed graph on the vertex set $[n] := \{1, 2, \cdots, n\}$, and let $\Gamma$ be its underlying undirected graph. We allow multiple edges, but a loop is not allowed. For a vertex $i$, the \emph{degree} of $i$ (denoted by $d_{i}$) is the number of edges incident to $i$ regardless of their directions. The \emph{multidegree} of $\overrightarrow{\Gamma}$ is defined by the degree sequence $(d_{1}, d_{2}, \cdots, d_{n})$ and denoted by $\mathbf{deg}\;\overrightarrow{\Gamma}$. Let $V_{\overrightarrow{\Gamma}} = V_{\Gamma}$ be the set of vertices, and let $E_{\overrightarrow{\Gamma}}$ be the set of directed edges. We define $\mathbf{deg}\; \Gamma := \mathbf{deg}\; \overrightarrow{\Gamma}$ and $E_{\Gamma}$ is the set of undirected edges in $\Gamma$. For any $I \subset [n]$, let $w_{I}$ be the number of edges connecting vertices in $I$. For notational simplicity, we set $w_{ij} = w_{\{i, j\}}$, which is the number of edges connecting $i$ and $j$. So $w_{I} = \sum_{i, j \in I}w_{ij}$. For $e \in E_{\overrightarrow{\Gamma}}$, $h(e) \in V_{\overrightarrow{\Gamma}}$ is the head and $t(e) \in V_{\overrightarrow{\Gamma}}$ is the tail. 

For each $\overrightarrow{\Gamma}$, let
\[
	Z_{\overrightarrow{\Gamma}} := 
	\prod_{e \in E(\Gamma)}(X_{t(e)}Y_{h(e)} - X_{h(e)}Y_{t(e)}).
\]
Then 
\[
	Z_{\overrightarrow{\Gamma}} \in \rH^{0}((\PP^{1})^{n}, 
	\cO(\mathbf{deg}\; \overrightarrow{\Gamma}))^{\SL_{2}}.
\]

We define the multiplication $\overrightarrow{\Gamma_{1}} \cdot \overrightarrow{\Gamma_{2}}$ of two graphs $\overrightarrow{\Gamma_{1}}$ and $\overrightarrow{\Gamma_{2}}$ by a graph with the vertex set $[n]$ and $E_{\overrightarrow{\Gamma_{1}} \cdot \overrightarrow{\Gamma_{2}}} := E_{\overrightarrow{\Gamma_{1}}} \sqcup E_{\overrightarrow{\Gamma_{2}}}$. Then $\mathbf{deg}\; \overrightarrow{\Gamma_{1}} \cdot \overrightarrow{\Gamma_{2}} = \mathbf{deg}\; \overrightarrow{\Gamma_{1}} + \mathbf{deg}\;\overrightarrow{\Gamma_{2}}$. Furthermore, 
\[
	Z_{\overrightarrow{\Gamma_{1}}\cdot \overrightarrow{\Gamma_{2}}} = 
	Z_{\overrightarrow{\Gamma_{1}}}\cdot 
	Z_{\overrightarrow{\Gamma_{2}}}.
\]
Note that if we reverse the direction of an edge $e \in E_{\overrightarrow{\Gamma}}$ and make a new graph $\overrightarrow{\Gamma}'$, then $Z_{\overrightarrow{\Gamma}'} = - Z_{\overrightarrow{\Gamma}}$.

The first fundamental theorem of invariant theory (\cite[Theorem 2.1]{Dol03}) says that the ring of $\SL_{2}$-invariants of $(\PP^{1})^{n}$ is generated by the polynomials $Z_{\overrightarrow{\Gamma}}$. 

\begin{definition}\label{def:graphicalalgebra}
The (total) \emph{graphical algebra} $R$ of order $n$ is defined by
\[
	R := \bigoplus_{L \in \mathrm{Pic}((\PP^{1})^{n})}
	\rH^{0}((\PP^{1})^{n}, L)^{\SL_{2}} =
	\bigoplus_{(a_{1}, a_{2}, \cdots, a_{n}) \in \ZZ_{\ge 0}^{n}}
	\rH^{0}((\PP^{1})^{n}, \cO(a_{1} ,a_{2}, \cdots, a_{n}))^{\SL_{2}}.
\]
\end{definition}

The support of $Z_{\overrightarrow{\Gamma}}$ is independent of the direction of each edge. Thus we may denote $\mathrm{Supp}(Z_{\overrightarrow{\Gamma}})$ by $D_{\Gamma}$. This $\SL_{2}$-invariant divisor on $(\PP^{1})^{n}$, or a Weil divisor on $(\PP^{1})^{n}\git \SL_{2}$ is called a \emph{graphical divisor}. The support of $Z_{ij}$ is denoted by $D_{ij}$.

The homogenous coordinate ring of the GIT quotient is a slice of $R$. Fix an effective linearization (a linearization with a nonempty semistable locus) $L \cong \cO(a_{1}, a_{2}, \cdots, a_{n})$. Then the homogeneous coordinate ring of $(\PP^{1})^{n}\git_{L}\SL_{2}$ is 
\[
	R_{L} := \bigoplus_{d \ge 0}\rH^{0}((\PP^{1})^{n}, L^{d})^{\SL_{2}}
	\subset R.
\]
The ideal of relations is explicitly described in \cite{HMSV09a, HMSV12}. 

From now on, we will use the symmetric linearization $\cO(1, 1, \cdots, 1)$ only. In this case, the GIT quotient $(\PP^{1})^{n}\git \SL_{2}$ is a projective variety with a natural $S_{n}$-action permuting the $n$ factors. If $n$ is odd, it is regular. If $n$ is even, there are ${n \choose n/2}/2$ non-regular closed points which are associated to closed orbits of two distinct points with multiplicities $n/2$. 

The generating set of $R_{L}$ has been well-understood since the 19th century by Kempe (\cite{Kem94}). A combinatorial description, including the relation ideal,  is given in \cite{HMSV09a}. We summarize the description here. 

\begin{theorem}[\protect{\cite{Kem94}, \cite[Theorem 2.3]{HMSV09a}}]\label{thm:genofR}
The homogeneous coordinate ring $R_{L}$ is generated by $Z_{\overrightarrow{\Gamma}}$ for $\overrightarrow{\Gamma}$ with $\mathbf{deg}\;\overrightarrow{\Gamma} = (\epsilon, \epsilon, \cdots, \epsilon)$ where $\epsilon = 2$ if $n$ is odd, and $\epsilon = 1$ if $n$ is even. 
\end{theorem}

Let $e_{i}$ be the $i$-th standard vector in $\ZZ^{n} \cong \mathrm{Pic}((\PP^{1})^{n})$. Each $D_{ij}$ on $(\PP^{1})^{n}\git \SL_{2}$ is the image of $V(Z_{ij})$ in $(\PP^{1})^{n}$ and $Z_{ij} \in \rH^{0}(\cO(e_{i}+e_{j}))$. Thus $\mathrm{Cl}((\PP^{1})^{n}\git \SL_{2})$ is identified with an index two sub-lattice of $\mathrm{Pic}((\PP^{1})^{n})$, generated by $\deg D_{ij} = e_{i} + e_{j}$. Note that a generator $D_{ij}$ of $\mathrm{Cl}((\PP^{1})^{n}\git \SL_{2})$ has a simple moduli theoretic interpretation. Indeed, 
\[
	D_{ij} = \{(p_{1}, p_{2}, \cdots, p_{n}) \in (\PP^{1})^{n}\git \SL_{2}
	 \;|\; p_{i} = p_{j}\}. 
\]
Let $D_{2} = \sum D_{ij}$. 

\begin{lemma}\label{lem:Picgenerator}
The $S_{n}$-invariant Picard group $\mathrm{Pic}((\PP^{1})^{n}\git \SL_{2})^{S_{n}}$ is generated by $\frac{1}{n-1}D_{2}$ (resp. $\frac{2}{n-1}D_{2}$) when $n$ is even (resp. odd). 
\end{lemma}

\begin{proof}
On $(\PP^{1})^{n}\git \SL_{2}$, we denote the descent of the line bundle $\cO(a_{1}, a_{2}, \cdots, a_{n})$ on $(\PP^{1})^{n}$ by $\overline{\cO}(a_{1}, a_{2}, \cdots, a_{n})$. By Kempf's descent lemma (\cite[Theorem 2.3]{DN89}), we are able to check when a line bundle on $(\PP^{1})^{n}$ descends to $(\PP^{1})^{n}\git \SL_{2}$. If $n$ is odd, it descends if and only if $\sum a_{i}$ is even. If $n$ is even, there is one extra constraint: For any $I \subset [n]$ with $|I| = n/2$, $\sum_{i \in I}a_{i} = \sum_{i \notin I}a_{i}$. The only line bundles satisfying this condition are the symmetric bundles $\cO(a, a, \cdots, a)$. Therefore if $n$ is odd, $\mathrm{Pic}((\PP^{1})^{n}\git \SL_{2})$ is isomorphic to an index two sub-lattice of $\mathrm{Pic}((\PP^{1})^{n}) \cong \ZZ^{n}$. If $n$ is even, $\mathrm{Pic}((\PP^{1})^{n}\git \SL_{2}) \cong \ZZ$. In particular, $\overline{\cO}(1, 1, \cdots, 1)$ (resp. $\overline{\cO}(2, 2, \cdots, 2)$) is an integral generator of $\mathrm{Pic}((\PP^{1})^{n}\git \SL_{2})^{S_{n}} \cong \ZZ$ when $n$ is even (resp. odd). Because $\cO(D_{2}) = \overline{\cO}(n-1, n-1, \cdots, n-1)$, we obtain the desired result. 
\end{proof}

%%%%%%%%%%%%%%%%%%%%%%%%%%%%%%%%%%%%%

\section{$G$-base-point-freeness}
\label{sec:Gbpf}

In this section, by using graphical algebra, we translate the $S_{n}$-invariant F-conjecture to a feasibility problem. We work over $\spec \ZZ$, unless there is an explicit assumption on the base scheme.

The following result explains an explicit connection between $\Mznb$ and $(\PP^{1})^{n}\git \SL_{2}$.

\begin{theorem}[\protect{\cite{Kap93b}, \cite[Theorem 4.1 and 8.3]{Has03}}]
There is a birational contraction morphism 
\[
	\pi : \Mznb \to (\PP^{1})^{n}\git\SL_{2}.
\]
\end{theorem}

For any $I \subset [n]$ with $2 \le I \le n/2$, let $B_{I} \subset \Mznb$ be the associated boundary divisor, and let $B_{i} := \sum_{|I| = i}B_{I}$. The image of $B_{ij}:= B_{\{i, j\}}$ is $D_{ij}$. For $I \subset [n]$ with $3 \le |I| < n/2$, $B_{I}$ is contracted by $\pi$, and its image is 
\[
	\pi(B_{I}) = \{(p_{1}, p_{2}, \cdots, p_{n})\;|\; p_{i} = p_{j} 
	\mbox{ for all } i, j \in I\}.
\]
If $p : ((\PP^{1})^{n})^{ss} \to (\PP^{1})^{n}\git \SL_{2}$ is the GIT quotient map, then $\pi(B_{I})$ is the image $p(W_{I})$ of $W_{I} := V(Z_{ij})_{i, j \in I} \subset ((\PP^{1})^{n})^{ss}$. 

When $n$ is even and $|I| = n/2$, $B_{I} = B_{I^{c}}$. Then $\pi(B_{I}) = \pi(B_{I^{c}})$ is an isolated singular point on $(\PP^{1})^{n}\git \SL_{2}$ and the associated closed orbit is 
\[
	\{(p_{1}, p_{2}, \cdots, p_{n})\;|\; p_{i} = p_{j} \mbox{ for all }i, j \in I 
	\mbox{ or } i, j \in I^{c}\}.
\]
Thus $\pi(B_{I})$ is the image $p(W_{I})$ of $W_{I} := V(Z_{ij})_{i, j \in I} \cap V(Z_{ij})_{i, j \in I^{c}}$. We denote $\pi(B_{I}) = p(W_{I})$ by $V_{I}$ for all $I$. 

By Theorem \ref{thm:genofR}, $D_{2}$ is very ample on $(\PP^{1})^{n}\git \SL_{2}$ since $\cO(D_{2}) = \overline{\cO}(n-1, n-1, \cdots, n-1)$. We have 
\begin{equation}\label{eqn:pullback}
	\pi^{*}(D_{2}) = \sum_{i \ge 2}^{\lfloor n/2 \rfloor}{i \choose 2}B_{i}
\end{equation}
(\cite[Lemma 5.3]{KM11}). Since $\pi$ is a regular contraction, the complete linear system $|\pi^{*}(D_{2})|$ is base-point-free on $\Mznb$. Indeed, $\pi^{*}(D_{2})$ is an extremal ray of $\nef(\Mznb)^{S_{n}}$. 

The following is a very simple but important observation. 

\begin{lemma}\label{lem:Fnefexpression}
Every non-trivial $S_{n}$-invariant F-nef $\QQ$-divisor on $\Mznb$ can be written uniquely as 
\[
	\pi^{*}(cD_{2}) - \sum_{i \ge 3}a_{i}B_{i}
\]
for some rational numbers $c > 0$ and $0 \le a_{i} < c{i \choose 2}$. Furthermore, it is integral if and only if $a_{i} \in \ZZ$ and $c \in \frac{1}{n-1}\ZZ$ (resp. $c \in \frac{2}{n-1}\ZZ$) when $n$ is even (resp. $n$ is odd). 
\end{lemma}

For notational simplicity, we set $a_{1} = a_{2} = 0$. 

\begin{proof}
Since F-nefness is defined formally, it is sufficient to prove the result over any algebraically closed field. In $\mathrm{N}^{1}(\Mznb)^{S_{n}}$, by \cite[Theorem 1.3]{KM13}, $\{B_{i}\}$ forms a $\QQ$-basis. So is $\{\pi^{*}(D_{2}), B_{i}\}_{3 \le i \le \lfloor n/2 \rfloor}$. Thus we have the existence and the uniqueness of the expression. Since $\mathrm{Eff}(\Mznb)^{S_{n}}$ is generated by $B_{i}$ for $2 \le i \le \lfloor n/2 \rfloor$ and every $S_{n}$-invariant F-nef divisor is big (\cite[Theorem 1.3]{KM13}), an $S_{n}$-invariant F-nef divisor is a strictly positive linear combination of $B_{i}$'s. From \eqref{eqn:pullback}, $a_{i} < c{i \choose 2}$ and $c > 0$. Let $F_{j}$ be any F-curve class whose associated partition has parts $\{1, 1, j, n-2-j\}$ for $1 \le j \le \lfloor n/2 \rfloor - 2$. Then since $F_{j}$ is contracted by $\pi$, 
\[
	0 \le F_{j} \cdot (\pi^{*}(cD_{2}) - \sum_{i \ge 3}a_{i}B_{i}) 
	= a_{j} + a_{j+2} - 2a_{j+1},
\]
so the sequence $\{a_{j}\}$ is convex. From $a_{1} = a_{2}= 0$, inductively we obtain $a_{i} \ge 0$ for all $i$.

The last assertion follows from Lemma \ref{lem:Picgenerator}, since each $B_{i}$ are all integral. 
\end{proof}

\begin{remark}
A natural question is whether the computational approach in this paper extends to non-symmetric F-nef divisors. It is unclear, at least to the authors, that every (non-necessarily $S_{n}$-invariant) F-nef divisor on $\Mznb$ can be written as a sum of divisors of the form in Lemma \ref{lem:Fnefexpression}.
\end{remark}

\begin{definition}
For a non-trivial integral $S_{n}$-invariant F-nef divisor $D = \pi^{*}(cD_{2}) - \sum_{i \ge 3}a_{i}B_{i}$, let $|cD_{2}|_{\mathbf{a}}$ be the sub linear system of $|cD_{2}|$ on $(\PP^{1})^{n}\git \SL_{2}$ consisting of the divisors whose multiplicity along $V_{I}$ is at least $a_{|I|}$.
\end{definition}

\begin{lemma}
Let $D = \pi^{*}(cD_{2}) - \sum_{i \ge 3}a_{i}B_{i}$ be an integral $S_{n}$-invariant F-nef divisor. The complete linear system $|D|$ is identified with $|cD_{2}|_{\mathbf{a}}$ on $(\PP^{1})^{n}\git \SL_{2}$.
\end{lemma}

\begin{proof}
Since $D$ is non-trivial, $c > 0$ by Lemma \ref{lem:Fnefexpression}. For any $E \in |cD_{2}|_{\mathbf{a}}$, $\pi^{*}E = E' + \sum_{i \ge 3}a_{i}B_{i}$ and $E' \in |D|$. Thus we have an injective map $|cD_{2}|_{\mathbf{a}} \to |D|$. Any divisor $F \in |D|$ can be written as $F + \sum_{i \ge 3}a_{i}B_{i} = \pi^{*}F'$ for some $F' \in |cD_{2}|$, and from the expression $F' \in |cD_{2}|_{\mathbf{a}}$. Thus we have a map $|D| \to |cD_{2}|_{\mathbf{a}}$. It is straightforward to see that they are inverses of each other.
\end{proof}

\begin{definition}
For a divisor $E \in |cD_{2}|_{\mathbf{a}}$, let $\widetilde{E} \in |D|$ be $\pi^{*}(E) - \sum_{i \ge 3}a_{i}B_{i}$, the divisor identified with $E$ via the isomorphism $|D| \cong |cD_{2}|_{\mathbf{a}}$.
\end{definition}

Note that $\widetilde{E}$ is \emph{not} the proper transform of $E$ in general. 

Although the non complete sub linear system $|cD_{2}|_{\mathbf{a}}$ has the key to understand the base-point-freeness or the semi-ampleness of $|D|$, still it is hard to analyze. We will define a sub linear system of $|cD_{2}|_{\mathbf{a}}$ which can be studied in purely combinatorial terms. 

Recall that for any $I \subset [n]$, $w_{I}$ is the number of edges connecting vertices in $I$. Recall also that for notational simplicity, we set $a_{2} = 0$. 

\begin{proposition}\label{prop:generatorsofD}
Let $D = \pi^{*}(cD_{2}) - \sum_{i \ge 3}a_{i}B_{i}$ be an integral $S_{n}$-invariant F-nef divisor. Let $D_{\Gamma}$ be a graphical divisor associated to a graph $\Gamma$. Then  $D_{\Gamma} \in |cD_{2}|_{\mathbf{a}}$ if and only if
\begin{enumerate}
\item $\mathbf{deg}\;\Gamma = (c(n-1), c(n-1), \cdots, c(n-1))$;
\item For every $I \subset [n]$ with $2 \le |I| \le n/2$, $w_{I} \ge a_{|I|}$;
\end{enumerate}
\end{proposition}

\begin{proof}
The condition (1) is exactly $D_{\Gamma} \in |cD_{2}|$, because $\cO(cD_{2}) = \overline{\cO}(c(n-1), c(n-1), \cdots, c(n-1))$. Note that for $3 \le |I| < n/2$, a general point of $V_{I}$ is smooth. Thus $D_{\Gamma}$ vanishes on $V_{I}$ with multiplicity at least $a_{|I|}$ if and only if $p^{*}(D_{\Gamma})$ vanishes on $W_{I} = V(Z_{ij})_{i, j \in I}$ with multiplicity at least $a_{|I|}$ if and only if $\Gamma$ has at least $a_{i}$ edges connecting vertices in $I$. 

When $n$ is even and $|I| = n/2$, $V_{I}$ is an isolated singular point and it is the image of $W_{I} = V(Z_{ij})_{i, j \in I} \cap V(Z_{ij})_{i, j \in I^{c}}$ on $((\PP^{1})^{n})^{ss}$. Suppose that $Z_{\Gamma}$ satisfies the following condition: 
\begin{center}
(C) For every $I \subset [n]$ with $|I| = n/2$, $w_{I} + w_{I^{c}} \ge 2a_{n/2}$. 
\end{center}
We claim that this condition is equivalent to the condition that multiplicity along $B_{n/2}$ is at least $a_{n/2}$. 

Consider the following commutative diagram:
\[
	\xymatrix{&(\mathrm{Bl}_{W_{I}}(\PP^{1})^{n})^{s}
	\ar[r]^{q} \ar[d]_{\tilde{p}} &
	 ((\PP^{1})^{n})^{ss}\ar[d]^{p}\\
	\Mznb \ar[r] & 
	\mathrm{Bl}_{W_{I}}(\PP^{1})^{n}\git \SL_{2} \ar[r]^{\bar{q}} &
	(\PP^{1})^{n}\git \SL_{2}}
\]
The vertical arrows are GIT quotients and $q$ is the blow-up along $W_{I}$. The superscript $ss$ (resp. $s$) denotes the semistable (resp. stable) locus. In \cite[Theorem 1.1]{KM11}, it was shown that Hassett's contraction $\pi : \Mznb \to (\PP^{1})^{n}\git \SL_{2}$ is decomposed into $\Mznb \to \mathrm{Bl}_{W_{I}}(\PP^{1})^{n}\git \SL_{2} \stackrel{\bar{q}}{\to} (\PP^{1})^{n}\git \SL_{2}$ and $\bar{q}$ is Kirwan's partial desingularization. 

For any $D_{G} \in |cD_{2}|$, $p^{*}(D_{G})$ is an $\SL_{2}$-invariant divisor on $((\PP^{1})^{n})^{ss}$. Its multiplicity along $W_{I}$ is at least  $2a_{|I|} = 2a_{n/2}$. Thus $q^{*}p^{*}(D_{G})$ has the multiplicity at least $2a_{n/2}$ along the exceptional divisor $E_{I}$. But since $-\mathrm{Id} \in \SL_{2}$ acts on $E_{I}$ nontrivially, $2E_{I}$ descends to $B_{I}$. Therefore the multiplicity along $B_{I}$ on $\mathrm{Bl}_{W_{I}}(\PP^{1})^{n}\git \SL_{2}$, which is equal to the multiplicity along $B_{I}$ on $\Mznb$, is at least $a_{n/2}$. The converse is similar.

Since the degree of each vertex is $c(n-1)$, $w_{I}= cn(n-1)/2 - \sum_{i \in I, j \in I^{c}}w_{ij} = w_{I^{c}}$. Thus we may reduce Condition (C) to (2). 
\end{proof}

\begin{definition}
Let $D = \pi^{*}(cD_{2}) - \sum_{i \ge 3}a_{i}B_{i}$ be a non-trivial $S_{n}$-invariant F-nef divisor with $c, ,a_{i} \in \ZZ$. Let $|cD_{2}|_{\mathbf{a}, G} \subset |cD_{2}|_{\mathbf{a}}$ be the sub linear system generated by $D_{\Gamma}$ satisfying two conditions in Proposition \ref{prop:generatorsofD}. Let $|D|_{G} \subset |D|$ be the sub linear system which is identified with $|cD_{2}|_{\mathbf{a}, G}$ via the identification
\[
	|D| \cong |cD_{2}|_{\mathbf{a}}.
\]
In other words, $|D|_{G}$ is generated by $\widetilde{D}_{\Gamma}$ for $\Gamma$ in Proposition \ref{prop:generatorsofD}.
\end{definition}

\begin{remark}
In general, $|D|_{G}$ is not equal to $|D|$. Equivalently, $|cD_{2}|_{\mathbf{a}, G}$ is not equal to $|cD_{2}|_{\mathbf{a}}$. See Example \ref{ex:bpfnotGbpf}. 
\end{remark}

The following lemma is straightforward.

\begin{lemma}\label{lem:baselocus}
Let $D = \pi^{*}(cD_{2}) - \sum_{i \ge 3}a_{i}B_{i}$ be an integral $S_{n}$-invariant F-nef divisor. Let $B = \bigcap_{I \in T}B_{I}$ be a boundary stratum indexed by a nonempty subset $T \subset \{I \subset [n]\;|\; 2 \le |I| \le \lfloor n/2 \rfloor\}$. Then a general point in $B$ is not in the support of $\widetilde{D}_{\Gamma} \in |D|_{G}$ if and only if for every $I \in T$ with $|I| \le n/2$, $w_{I} = a_{|I|}$.
\end{lemma}

\begin{definition} \label{def:Gsemiample}
An integral $S_{n}$-invariant F-nef divisor $D = \pi^{*}(cD_{2}) - \sum_{i \ge 3}a_{i}B_{i}$ is called \emph{$G$-base-point-free} if $|D|_{G}$ is base-point-free. It is \emph{$G$-semi-ample} if $|mD|_{G}$ is base-point-free for some $m \gg 0$. 
\end{definition}

This sub linear system is particularly nice because the base locus can be described in a combinatorial way. 

\begin{lemma}\label{lem:baselocus2}
Let $D = \pi^{*}(cD_{2}) - \sum_{i \ge 3}a_{i}B_{i}$ be an integral $S_{n}$-invariant F-nef divisor. Then the base locus $\mathrm{Bs}(|D|_{G})$ is a union of closures of boundary strata. 
\end{lemma}

\begin{proof}
For $\widetilde{D}_{\Gamma} \in |D|_{G}$, 
\[
	\mathrm{Supp}(\widetilde{D}_{\Gamma}) = \bigcup_{I}B_{I}
\]
where the sum is taken over all $I$ where the number of edges connecting vertices in $I$ is strictly larger than $a_{|I|}$. Since $|D|_{G}$ is generated by $\widetilde{D}_{\Gamma}$,
\[
	\mathrm{Bs}(|D|_{G}) = \bigcap_{\widetilde{D}_{\Gamma} \in |D|_{G}}
	\mathrm{Supp}(\widetilde{D}_{\Gamma}).
\]
\end{proof}

\begin{proposition}\label{prop:Gbpf}
Let $D = \pi^{*}(cD_{2}) - \sum_{i \ge 3}a_{i}B_{i}$ be an integral $S_{n}$-invariant F-nef divisor. Then $D$ is $G$-base-point-free if and only if for every F-point $F = \cap_{I \in T}B_{I}$, there is a graph $\Gamma$ such that 
\begin{enumerate}
\item $\mathbf{deg}\; \Gamma = (c(n-1), c(n-1), \cdots, c(n-1))$;
\item For each $I \subset [n]$ with $2 \le |I| \le n/2$, $w_{I} \ge a_{|I|}$;
\item For each $J \in T$ with $|J| \le n/2$, $w_{J} = a_{|J|}$;
\end{enumerate}
\end{proposition}

\begin{proof}
The above condition implies that the base locus of $|D|_{G}$ does not contain any F-point. Since $\mathrm{Bs}(|D|_{G})$ is a union of boundary strata by Lemma \ref{lem:baselocus2}, if it is non-empty, then there must be at least on F-point in it. Thus $|D|_{G}$ is base-point-free.
\end{proof}

Since the $G$-base-point-freeness implies the base-point-freeness, Proposition \ref{prop:Gbpf} provides a purely combinatorial/computational sufficient condition for being a base-point-free divisor. 

Note that sums and scalar multiples of graphical divisors are graphical divisors, too. So it is straightforward to see that for any two $G$-base-point-free (resp. $G$-semi-ample) divisors $D$ and $D'$, $D+D'$ and their nonnegative scalar multiples are all $G$-base-point-free (resp. $G$-semi-ample). 

\begin{definition}\label{def:GS}
Let $\mathrm{GS}(\Mznb)$ be the convex cone generated by $G$-semi-ample divisors in $\mathrm{N}^{1}(\Mznb)^{S_{n}}$.
\end{definition}

We have the following obvious implications:
\begin{equation}\label{eqn:implications}
	\mbox{$G$-base-point-free} \Rightarrow
	\mbox{$G$-semi-ample} \Rightarrow
	\mbox{semi-ample} \Rightarrow
	\mbox{nef} \Rightarrow \mbox{$F$-nef}.
\end{equation}

\begin{theorem}\label{thm:G-semiamplecone}
The cone $\mathrm{GS}(\Mznb)$ of $G$-semi-ample divisors of $\Mznb$ is closed and polyhedral.
\end{theorem}

\begin{proof}

We may consider a multigraph $\Gamma$ as a graph weighting $w : E_{K_{n}} \to \QQ$ where $E_{K_{n}}$ is the set of edges on the complete graph $K_{n}$, by setting $w(\overline{ij}) = w_{ij}$, the number of edges between $i$ and $j$. Consider $V := \QQ^{{n \choose 2}}$ with coordinates $\{w_{ij}\}_{1 \le i < j \le n}$. This space can be regarded as the space of graph weightings. By representing any non-trivial $S_{n}$-invariant F-nef divisor in the form $\pi^{*}(cD_{2}) - \sum_{i \ge 3}a_{i}B_{i}$, we may identify $\mathrm{N}^{1}(\Mznb)^{S_{n}}$ with $\QQ^{\lfloor n/2 \rfloor - 1}$ whose coordinates are $(c, a_{i})_{3 \le i \le \lfloor n/2\rfloor}$. For each F-point $F = \bigcap_{J \in T}B_{J}$, we can define a polyhedral cone $Q(n, F) \subset V \times \mathrm{N}^{1}(\Mznb)^{S_{n}}$ by the following inequalities and equations:
\begin{enumerate}
\item $c, a_{i}, w_{ij} \ge 0$;
\item $\sum_{j \ne i} w_{ij} = c(n-1)$;
\item $\sum_{i,j \in I}w_{ij} \ge a_{|I|}$ for each $I$ with $3 \le |I| \le n/2$;
\item $\sum_{i,j \in J}w_{ij} = a_{|J|}$ for each $J \in T$ with $|J| \le n/2$.
\end{enumerate}

Let $\rho : V \times \mathrm{N}^{1}(\Mznb)^{S_{n}}\to \mathrm{N}^{1}(\Mznb)^{S_{n}}$ be the projection defined by 
\begin{equation}\label{eqn:projection}
	\rho(w_{ij}, c, a_{i}) = \pi^{*}(cD_{2}) - \sum_{i \ge 3}a_{i}B_{i} 
	= (c, a_{i}).
\end{equation}
For an integral $S_{n}$-invariant F-nef divisor $D = \pi^{*}(cD_{2}) - \sum_{i \ge 3}a_{i}B_{i}$, $\mathrm{Bs}(|D|_{G})$ does not contain an F-point $F := \bigcap_{J \in T}B_{J}$ if and only if $\rho^{-1}(D) \cap Q(n, F)$ has an integral point. So $\mathrm{Bs}(|mD|_{G})$ does not contain $F$ for some $m$ if and only if $\rho^{-1}(D) \cap Q(n, F)$ has a rational point, if and only if $D \in \rho(Q(n, F))$. 

Therefore $|mD|_{G}$ is base-point-free for some $m \gg 0$ if and only if $D \in \bigcap_{F}\rho(Q(n, F))$ where the intersection is taken over all F-points. In particular, a priori, the $G$-semi-ample cone is the intersection of the $S_{n}$-invariant F-nef cone and $\bigcap_{F}\rho(Q(n, F))$. Therefore it is polyhedral and closed.

Indeed, $\bigcap_{F}\rho(Q(n, F))$ is a subcone of the $S_{n}$-invariant F-nef cone. For any boundary stratum $B = \bigcap_{J \in T}B_{J}$, we can define a cone $Q(n, B) \subset V \times \mathrm{N}^{1}(\Mznb)^{S_{n}}$ in the same way. If $B$ and $B'$ are two boundary strata such that $B \subset B'$, then $Q(n, B) \subset Q(n, B')$. Therefore the intersection $\bigcap_{B}\rho(Q(n, B))$ for all boundary strata is equal to $\bigcap_{F}\rho(Q(n, F))$ where the intersection is taken over F-points. In particular, if $D$ is a divisor in $\bigcap_{F}\rho(Q(n, F))$, then for each F-curve $F_{I}$, there is a section in $\rH^{0}(\cO(mD))$ for some $m \gg 0$ which is nonzero on $F_{I}$. This implies $D\cdot F_{I} \ge 0$ and $D$ is F-nef.
\end{proof}

Thus we obtain a polyhedral lower bound of $\mathrm{Nef}(\Mznb)^{S_{n}}$. 

By the proof of Theorem \ref{thm:G-semiamplecone}, we obtain the following corollary, which provides a computational approach to the $S_{n}$-invariant F-conjecture. 

\begin{corollary}\label{cor:polyhedralcone}
Let $\rho: V  \times \rN^{1}(\Mznb)^{S_{n}} \to \rN^{1}(\Mznb)^{S_{n}}$ be the projection in \eqref{eqn:projection}. Then an $S_{n}$-invariant F-nef divisor $\pi^{*}(cD_{2}) - \sum a_{i}B_{i}$ is G-semi-ample (hence semi-ample) if and only if $D \in \bigcap_{F} \rho(Q(n, F))$.  
\end{corollary}

%%%%%%%%%%%%%%%%%%%%%%%%%%%%%%%%%%%%

\section{Computational results}\label{sec:computation}

In this section we list several computational results. The calculations can be found on the webpage \cite{MS16}.

\begin{theorem}\label{thm:semiamplecomputation}
For $n \le \checkedn$, over $\spec \ZZ$, the $S_{n}$-invariant F-nef cone coincides with the $G$-semi-ample cone. 
\end{theorem}

\begin{proof}
Since $\mathrm{GS}(\Mznb)$ is a convex subcone of the
$S_{n}$-invariant F-nef cone, it is sufficient to show that every
integral generator of an extremal ray of the $S_{n}$-invariant F-nef
cone is $G$-semi-ample. By Corollary \ref{cor:polyhedralcone}, it is
sufficient to show the feasibility of the polytope $\rho^{-1}(D) \cap
Q(n, F)$ for each integral generator $D$ and an F-point $F$. By using
Sage \cite{Sage} and Gurobi \cite{Gur16}, we checked that for $n \le \checkedn$, such a polytope is nonempty. 
\end{proof}

Indeed a more efficient result is true. By checking $G$-base-point-freeness of the Hilbert basis of $S_{n}$-invariant F-nef cone, we obtain the following result. 

\begin{theorem}\label{thm:bpfcomputation}
\begin{enumerate}
\item For $n \le \bpfcheckedn$, over $\spec \ZZ$, for any integral $S_{n}$-invariant F-nef divisor $D$, $2D$ is $G$-base-point-free. 
\item For $n \le 11$ or $13$, over $\spec \ZZ$, for any integral $S_{n}$-invariant F-nef divisor $D$, $D$ is $G$-base-point-free. 
\end{enumerate}
\end{theorem}

\begin{remark}
This computation was faster than we predicted. To check the $G$-base-point-freeness of a divisor $D$ we need to do the following computation.
\begin{enumerate}
\item Take a representative of an F-point $F$ from each $S_{n}$-orbit. Let $P$ be the set of the representatives. 
\item For each $F \in P$, compute the nonemptyness of $Q(n, F)$.
\end{enumerate}
But when $n$ is small, for most of integral divisors, $\cap_{F \in P}Q(n, F)$ is nonempty. Thus it is sufficient to solve the feasibility problem once per each divisor. Even when this is not the case, the number of feasibility problems we need to solve is significantly small compare to the cardinality of $P$.
\end{remark}

\begin{conjecture}
For any integral $S_{n}$-invariant F-nef divisor $D$, $2D$ is $G$-base-point-free. In particular, $2D$ is base-point-free.
\end{conjecture}

\begin{remark}\label{rem:bpf}
Most of the integral divisors are $G$-base-point-free. More precisely, for $n = 12$, there are only two integral $S_{n}$-invariant F-nef divisors which are not base-point-free. For $n = 14, 15$, there is only one for each $n$. 
\end{remark}

\begin{example}\label{ex:bpfnotGbpf}
Let $n = 12$. Consider the divisor class
\[
	D = \frac{1}{11}(4B_{2} + 12B_{3} + 13B_{4} + 18B_{5} + 16B_{6})
	= \pi^{*}(\frac{4}{11}D_{2}) - B_{4} - 2B_{5} - 4B_{6}.
\]
Here we give an example of an integral $S_{n}$-invariant base-point-free divisor which is not G-base-point-free. Then the base locus $\mathrm{Bs}(|D|_{G})$ contains the $S_{12}$-orbit of an F-point
\[
	F = B_{\{1, 2\}} \cap B_{\{1, 2, 3\}} \cap B_{\{4, 5\}} \cap 
	B_{\{4, 5, 6\}} \cap B_{\{1, 2, 3, 4, 5, 6\}} \cap B_{\{7, 8\}} \cap 
	B_{\{7, 8, 9\}} \cap B_{\{10, 11\}} \cap B_{\{10, 11, 12\}}.
\]
One can check that the locus of curves having four tails with three marked points is the base locus of $|D|_{G}$. The base locus is isomorphic to a disjoint union of 15400 disjoint unions of $(\overline{\mathrm{M}}_{0, 4})^{5}$. 

On the other hand, by using a computer and Kapranov's model, we
constructed a divisor $E \in |D|$ such that $F \notin E$. Therefore
$|D|_{G} \ne |D|$.  See \cite{MS16}.
\end{example}

Now the very ampleness of $S_{n}$-invariant ample divisors is an immediate consequence of the result of Keel and Tevelev. 

\begin{theorem}\label{thm:veryample}
Let $n \le \bpfcheckedn$. Over any algebraically closed field, for every integral $S_{n}$-invariant ample divisor $A$ on $\Mznb$, $2A$ is very ample.
\end{theorem}

\begin{proof}
By \cite[Theorem 1.5]{Tev07} or \cite[Theorem 1.1]{KT09}, the log canonical divisor $K_{\Mznb}+B$ is very ample. Note that for every F-curve $F$, $(K_{\Mznb}+B) \cdot F = 1$. So if $A$ is an $S_{n}$-invariant integral ample divisor, then $A - (K_{\Mznb}+B)$ is an $S_{n}$-invariant nef divisor, so $2(A - (K_{\Mznb}+B))$ is base-point-free by Theorem \ref{thm:bpfcomputation}. Therefore
\[
	2A = 2(A-(K_{\Mznb}+B)) + 2(K_{\Mznb}+B)
\]
is a sum of a base-point-free divisor and a very ample divisor, which is very ample.
\end{proof}

\begin{remark}\label{rmk:veryamplelown}
By Remark \ref{rem:bpf}, when $n \le 11$ or $n = 13$, every integral $S_{n}$-invariant nef divisor is base-point-free. So for those $n$, every integral $S_{n}$-invariant ample divisor is very ample.
\end{remark}

%%%%%%%%%%%%%%%%%%%%%%%%%%%%%%%%%%%%%

\section{Comparison to other cones}\label{sec:comparison}

There are several lower bounds of $\mathrm{Nef}(\Mznb)$. In this section we compare $S_{n}$-invariant part of them with the cone $\mathrm{GS}(\Mznb)$ of $G$-semi-ample divisors. 

In \cite{GM12}, Gibney and Maclagan defines a lower bound of $\mathrm{Nef}(\Mznb)$, by using an embedding of $\Mznb$ into a non-proper toric variety. Let $X_{\Delta}$ be a toric variety whose associated fan is $\Delta \subset \RR^{n}$. Suppose that there is a projective toric variety $X_{\Sigma}$ with $\Delta \subset \Sigma$. 

\begin{definition}
The cone $\mathcal{G}_{\Delta} \subset \mathrm{Pic}(X_{\Delta})_{\QQ}$ is the semi-ample cone of $X_{\Delta}$.
\end{definition}

Let $Y$ be a projective variety with an embedding $i : Y \hookrightarrow X_{\Delta}$. Then we obtain a lower bound 
\[
	i^{*}(\mathcal{G}_{\Delta}) \subset \nef(Y).
\]

The cone $i^{*}(\mathcal{G}_{\Delta})$ is polyhedral and can be described combinatorially.  These properties follow from the corresponding properties of $\mathcal{G}_{\Delta}$:

\begin{proposition}[\protect{\cite[Proposition 2.3]{GM12}}]
\label{prop:coneG}
Let $D = \sum_{i \in \Delta(1)}a_{i}D_{i}\in \mathrm{Pic}(X_{\Delta})$. Then the following are equivalent:
\begin{enumerate}
\item $D \in \mathcal{G}_{\Delta}$;
\item $D \in \bigcap_{\sigma \in \Delta}\mathrm{pos}(D_{i}\;|\; i \notin \sigma)$;
\item there is a piecewise linear convex function $\psi : N_{\RR} \to \RR$ and $\psi(v_{i}) = a_{i}$ where $v_{i}$ is the first integral vector in the ray $i$. 
\item $D \in \bigcup_{\Sigma}i^{*}_{\Sigma}(\nef(X_{\Sigma}))$ where the union is over all projective toric varieties $X_{\Sigma}$ with $\Delta \subset \Sigma$ and $i_{\Sigma} : X_{\Delta} \to X_{\Sigma}$ is the inclusion. Furthermore, we may assume that $\Delta(1) = \Sigma(1)$. 
\end{enumerate}
\end{proposition}

It is well-known that $\Mznb$ can be embedded into a non-proper toric variety $X_{\Delta}$ where $\Delta$ is the space of phylogenetic trees. Thus we obtain a lower bound $i^{*}(\mathcal{G}_{\Delta})$ of $\mathrm{Nef}(\Mznb)$. 

On the other hand, in \cite{Fed14b}, Fedorchuk introduced another combinatorial notion of \emph{boundary semi-ampleness}. 

\begin{definition}
Let $D$ be a divisor on $\Mznb$. $D$ is \emph{boundary semi-ample} if for every $x \in \Mznb$, there exists an effective boundary $\QQ$-divisor $E \in |D|$ such that $x \notin \mathrm{Supp}(E)$. 
\end{definition}

The following result was pointed out by Fedorchuk. 

\begin{proposition}
Suppose that $D$ is an $S_{n}$-invariant divisor on $\Mznb$. Then the following three conditions are equivalent:
\begin{enumerate}
\item $D \in \mathrm{GS}(\Mznb)$;
\item $D$ is boundary semi-ample;
\item $D \in i^{*}(\mathcal{G}_{\Delta})$.
\end{enumerate}
\end{proposition}

\begin{proof}
Clearly a $G$-semi-ample divisor $D$ is boundary semi-ample. 

Because $\mathrm{Pic}(X_{\Delta}) \cong \mathrm{Pic}(\Mznb)$ and each boundary divisor $B_{I}$ is a restriction of a toric boundary, from item (2) of Proposition \ref{prop:coneG}, it is straightforward to see that $D$ is boundary semi-ample if and only if $D \in i^{*}(\mathcal{G}_{\Delta})$. 

So it is sufficient to show that every $S_{n}$-invariant boundary semi-ample divisor is $G$-semi-ample. Suppose that $D$ is boundary semi-ample and $S_{n}$-invariant. By \cite[Lemma 3.2.3]{Fed14b}, 
\[
	D = \mathcal{D}(\ZZ_{n}, f; (1)_{n}) 
	= f(1)\psi - \sum_{2 \le i \le n/2}f(i)B_{i}
\]
for some symmetric F-nef function $f : \ZZ_{n} \to \QQ$ (See \cite[Section 3]{Fed14b} for the notation). On the other hand, $D$ can be uniquely written as $\pi^{*}(cD_{2}) - \sum_{i \ge 3}a_{i}B_{i}$. If we set $a_{1} = a_{2} = 0$, then by using $\pi^{*}(D_{2}) = \frac{c(n-1)}{2}(-\psi+\sum_{i\ge 2} iB_{i})$, one can see that $f(i) = (-ci(n-1) + a_{i})/2$. 

Then by \cite[Lemma 2.3.3]{Fed14b}, for an F-point $F = \cap_{I \in T}B_{I}$, there is an effective sum of boundary $E \in |D|$ such that $F \notin \mathrm{Supp}(E)$ if and only if there is a graph weighting $m : E_{K_{n}} \to \QQ$ such that
\begin{enumerate}
\item $\sum_{j \ne i}m(ij) = f(1)$;
\item $\sum_{i \in I, j \in I^{c}}m(ij) \ge f(i)$;
\item For $I \in T$, $\sum_{i \in I, j \in I^{c}}m(ij) = f(i)$.
\end{enumerate}
Now set $w(ij) := -2m(ij)$. Then from (1), $\sum_{j \ne i}w(ij) = -2(\sum_{j \ne i}m(ij)) = c(n-1)$. Note that if $|I| = i$, then $2\sum_{i, j \in I}w(ij) + \sum_{i \in I, j \in I^{c}}w(ij) = ic(n-1)$. So 
\[
\begin{split}
	ic(n-1) &= 2\sum_{i, j \in I}w(ij) + \sum_{i \in I, j \in I^{c}}w(ij)
	= 2\sum_{i, j \in I}w(ij) - 2\sum_{i \in I, j \in I^{c}}w(ij)
	\le 2\sum_{i, j \in I}w(ij) - 2f(i)\\
	&= 2\sum_{i,j\in I}w(ij) + ci(n-1) - a_{i},
\end{split}
\]
and this implies $\sum_{i, j \in I}w(ij) \ge a_{i}$. Similarly, Item (3) is $\sum_{i, j \in I}w(ij) = a_{i}$ for all $I \in T$. Therefore we obtain precisely, the inequalities and equalities for $G$-semi-ampleness. Therefore $D$ is $G$-semi-ample. 
\end{proof}

\newcommand{\etalchar}[1]{$^{#1}$}

%
%
%\bibliographystyle{alpha}
%%\newcommand{\etalchar}[1]{$^{#1}$}
%\bibliography{Library}

\begin{thebibliography}{HMSV12}

\bibitem[CHS09]{CHS09}
Izzet Coskun, Joe Harris, and Jason Starr.
\newblock The ample cone of the {K}ontsevich moduli space.
\newblock {\em Canad. J. Math.}, 61(1):109--123, 2009.

\bibitem[CS06]{CS06}
Izzet Coskun and Jason Starr.
\newblock Divisors on the space of maps to {G}rassmannians.
\newblock {\em Int. Math. Res. Not.}, pages Art. ID 35273, 25, 2006.

\bibitem[CT13]{CT13}
Ana-Maria Castravet and Jenia Tevelev.
\newblock Hypertrees, projections, and moduli of stable rational curves.
\newblock {\em J. Reine Angew. Math.}, 675:121--180, 2013.

\bibitem[CT15]{CT15}
Ana-Maria Castravet and Jenia Tevelev.
\newblock {$\overline{M}_{0,n}$} is not a {M}ori dream space.
\newblock {\em Duke Math. J.}, 164(8):1641--1667, 2015.

\bibitem[DGJ14]{DGJ14}
Brent Doran, Noah Giansiracusa, and David Jensen.
\newblock A simplicial approach to effective divisors in $\overline{M}_{0,n}$.
\newblock arXiv:1401.0350, 2014.

\bibitem[DN89]{DN89}
J.-M. Drezet and M.~S. Narasimhan.
\newblock Groupe de {P}icard des vari{\'e}t{\'e}s de modules de fibr{\'e}s
  semi-stables sur les courbes alg{\'e}briques.
\newblock {\em Invent. Math.}, 97(1):53--94, 1989.

\bibitem[Dol03]{Dol03}
Igor Dolgachev.
\newblock {\em Lectures on invariant theory}, volume 296 of {\em London
  Mathematical Society Lecture Note Series}.
\newblock Cambridge University Press, Cambridge, 2003.

\bibitem[Fed14]{Fed14b}
Maksym Fedorchuk.
\newblock Semiampleness criteria for divisors on $\overline{M}_{0,n}$.
\newblock arXiv:1407.7839, 2014.

\bibitem[Gib09]{Gib09}
Angela Gibney.
\newblock Numerical criteria for divisors on {$\overline M_g$} to be ample.
\newblock {\em Compos. Math.}, 145(5):1227--1248, 2009.

\bibitem[GK16]{GK16}
Jos{{\'e}}~Luis Gonz{{\'a}}lez and Kalle Karu.
\newblock Some non-finitely generated {C}ox rings.
\newblock {\em Compos. Math.}, 152(5):984--996, 2016.

\bibitem[GKM02]{GKM02}
Angela Gibney, Sean Keel, and Ian Morrison.
\newblock Towards the ample cone of {$\overline M_{g,n}$}.
\newblock {\em J. Amer. Math. Soc.}, 15(2):273--294, 2002.

\bibitem[GM12]{GM12}
Angela Gibney and Diane Maclagan.
\newblock Lower and upper bounds for nef cones.
\newblock {\em Int. Math. Res. Not. IMRN}, (14):3224--3255, 2012.

\bibitem[Gur16]{Gur16}
Gurobi Optimization, Inc.
\newblock Gurobi Optimizer Reference Manual.
\newblock http://www.gurobi.com, 2016.

\bibitem[Has03]{Has03}
Brendan Hassett.
\newblock Moduli spaces of weighted pointed stable curves.
\newblock {\em Adv. Math.}, 173(2):316--352, 2003.

\bibitem[HKL16]{HKL16}
Juergen Hausen, Simon Keicher, and Antonio Laface.
\newblock On blowing up the weighted projective space.
\newblock arXiv:1608.04542, 2016.

\bibitem[HMSV09]{HMSV09a}
Benjamin Howard, John Millson, Andrew Snowden, and Ravi Vakil.
\newblock The equations for the moduli space of {$n$} points on the line.
\newblock {\em Duke Math. J.}, 146(2):175--226, 2009.

\bibitem[HMSV12]{HMSV12}
Benjamin Howard, John Millson, Andrew Snowden, and Ravi Vakil.
\newblock The ideal of relations for the ring of invariants of {$n$} points on
  the line.
\newblock {\em J. Eur. Math. Soc. (JEMS)}, 14(1):1--60, 2012.

\bibitem[Kap93]{Kap93b}
M.~M. Kapranov.
\newblock Chow quotients of {G}rassmannians. {I}.
\newblock In {\em I. {M}. {G}elfand {S}eminar}, volume~16 of {\em Adv. Soviet
  Math.}, pages 29--110. Amer. Math. Soc., Providence, RI, 1993.

\bibitem[Kem94]{Kem94}
A.~B. Kempe.
\newblock On {R}egular {D}ifference {T}erms.
\newblock {\em Proc. London Math. Soc.}, S1-25(1):343, 1894.

\bibitem[KM11]{KM11}
Young-Hoon Kiem and Han-Bom Moon.
\newblock Moduli spaces of weighted pointed stable rational curves via {GIT}.
\newblock {\em Osaka J. Math.}, 48(4):1115--1140, 2011.

\bibitem[KM13]{KM13}
Se{{\'a}}n Keel and James McKernan.
\newblock Contractible extremal rays on {$\overline M_{0,n}$}.
\newblock In {\em Handbook of moduli. {V}ol. {II}}, volume~25 of {\em Adv.
  Lect. Math. (ALM)}, pages 115--130. Int. Press, Somerville, MA, 2013.

\bibitem[KT09]{KT09}
Sean Keel and Jenia Tevelev.
\newblock Equations for {$\overline M_{0,n}$}.
\newblock {\em Internat. J. Math.}, 20(9):1159--1184, 2009.

\bibitem[MS16]{MS16}
Han-Bom Moon and David Swinarski.
\newblock Supplementary files for ``On the $S_{n}$-invariant F Conjecture''.
\newblock http://faculty.fordham.edu/dswinarski/G-semiampleness, 2016.

\bibitem[Opi16]{Opi16}
Morgan Opie.
\newblock Extremal divisors on moduli spaces of rational curves with marked
  points.
\newblock {\em Michigan Math. J.}, 65(2):251--285, 2016.

\bibitem[S{\etalchar{+}}]{Sage}
W.~A. Stein et~al.
\newblock {\em Sage}.
\newblock The Sage Development Team, {S}age {M}athematics {S}oftware ({V}ersion
  6.5).

\bibitem[Tev07]{Tev07}
Jenia Tevelev.
\newblock Compactifications of subvarieties of tori.
\newblock {\em Amer. J. Math.}, 129(4):1087--1104, 2007.

\bibitem[Ver02]{Ver02}
Peter Vermeire.
\newblock A counterexample to {F}ulton's conjecture on {$\overline M_{0,n}$}.
\newblock {\em J. Algebra}, 248(2):780--784, 2002.

\end{thebibliography}

\end{document}